\documentclass{amsart}
\usepackage{newtxtext,newtxmath}
\usepackage{graphicx}
\usepackage{hyperref}

\usepackage{amsmath,amssymb,mathscinet}
\usepackage{amsthm,mathtools}
\newtheorem{theorem}{Theorem}[section]
\newtheorem{prop}[theorem]{Proposition}
\newtheorem{lemma}[theorem]{Lemma}
\newtheorem{example}[theorem]{Example}
\newtheorem{corollary}[theorem]{Corollary}

\theoremstyle{definition}

\newtheorem{problem}[theorem]{Problem}

\theoremstyle{remark}
\newtheorem{remark}[theorem]{Remark}

\numberwithin{equation}{section}
\newcommand{\htop}{h_{\mathrm{top}}}

\newcommand{\R}{{\mathbb R}}          
\newcommand{\N}{{\mathbb N}}           
\newcommand{\II}{{\mathbf I}}          

\newcommand{\ww}{\widetilde}

\newcommand{\Int}{\operatorname{int}}           
\newcommand{\id}{\operatorname{id}}             

\begin{document}

\title{Subspaces of interval maps related to the topological entropy}
\author{Xiaoxin Fan, Jian Li, Yini Yang and Zhongqiang Yang}
\address{Department of Mathematics,
Shantou University, Shantou, Guangdong, 515063, China
P.R.}
\email{[X. Fan]14xxfan@alumni.stu.edu.cn}
\email{[J. Li]lijian09@mail.ustc.edu.cn}
\email{[Y.Yang]ynyangchs@foxmail.com}
\email{[Z.Yang]zqyang@stu.edu.cn}

\begin{abstract}%
	For $a\in [0,+\infty)$,  the function space $E_{\geq a}$ ($E_{>a}$; $E_{\leq a}$; $E_{<a}$) of  all continuous maps from $[0,1]$ to itself whose topological entropies are larger than or equal to $a$ (larger than  $a$; smaller than or equal to $a$; smaller than $a$) with the supremum metric is investigated.
	It is shown that the spaces $E_{\geq a}$ and $E_{>a}$ are homeomorphic to the Hilbert space $l_2$ and the spaces $E_{\leq a}$ and  $E_{<a}$ are contractible. Moreover, the subspaces of $E_{\leq a}$ and  $E_{<a}$ consisting of all piecewise monotone maps are homotopy dense in them, respectively.
\end{abstract}

\thanks{X. Fan and Z. Yang was supported in part by the NNSF of China (11471202, 11971287);
	J. Li (corresponding author) and Y. Yang  was supported in part by the NNSF of China (1771264, 11871188)
	and NSF of Guangdong Province (2018B030306024)}

\maketitle

\section{Introduction}
One of the central topics in the study of infinite-dimensional topology
is that which kinds of function spaces  are homeomorphic to
the separable infinite dimensional Hilbert space $l_2$ or its well-behaved subspaces.
The well-known Anderson-Kadec's theorem states that
the countable infinite product $\R^\N$ of lines is homeomorphic to $l_2$, see \cite{Anderson,Kadec}. Using this result,
it was proved that the function space of real valued maps of an infinite compact metric space with the supremum metric is homeomorphic to $l_2$.
See \cite{Chigogidze-2001,van-book-1989,van-book-2001} for more on this topic. Moreover, in \cite{Dobrowolski-Marciszewski-1991}, the authors proved that the function space of real valued maps of an
infinite countable metric space with the topology of pointwise convergence is homeomorphic to  the subspace $c_0=\{(x_n)\in \R^\N:\lim\limits_{n\to\infty} x_n=0\}$ of  $\R^\N$.
In a series of papers, the fourth named author of the present paper and his coauthors gave a condition for the continuous functions from a k-space to $\II=[0,1]$ with the Fell hypergraph topology being homeomorphic to $c_0$, see \cite{Yang-Yang-Zhen,Yang-Chen-Zheng, Yang-Yan-2014, Yang-Zheng-Chen}.

In the study of dynamical systems, some function spaces naturally appear.
The group of measure preserving transformations of the unit interval
equipped with the weak topology is homeomorphic
to $l_2$ (see \cite{D86} and \cite{N91}).
Recently, in \cite{Kolyada-2151}
Kolyada et al. proposed the study of dynamical topology:
investigating the topological properties of spaces of maps that can
be described in dynamical terms.
They showed in \cite{Kolyada-2151} that the space of transitive
interval maps is contractible and
uniformly locally arcwise connected, see also \cite{Kolyada-137} for more detailed results. In \cite {Grinc}, Grinc et al. discussed
some topological properties of subspaces of interval maps related to the periods of periodic points.

In this paper, we will follow the idea in \cite{Kolyada-2151}
to study subspaces of interval maps 
related to the topological entropy.
Let  $\II=[0,1]$ and $C(\II)$ be the collection of
continuous maps on $\II$ with the supremum metric $d$.
For each $f\in C(\II)$, denote by $\htop(f)$ the topological entropy of $f$.
For any  $a\in [0,+\infty]$, let
\begin{align*}
E_{\geq a}&=\{f\in C(\II):\htop(f)\geq a\};\\
E_{> a}&=\{f\in C(\II):\htop(f)> a\};\\
E_{\leq a}&=\{f\in C(\II):\htop(f)\leq a\};\\
E_{< a}&=\{f\in C(\II):\htop(f)< a\}.
\end{align*}
A map $f\in C(\II)$ is said to be \textbf{piecewise monotone} if there exist $0=t_0<t_1<\cdots<t_n=1$ such that $f|_{[t_{i-1},t_i]}$ is monotone for every $i=1,2,\dotsc,n$.
Similarly, we can define a map to be \textbf{piecewise linear}.
We use $C^{PM}(\II)$ to denote the set of all piecewise monotone continuous maps on $\II$ and $$E^{PM}_{\leq a}=E_{\leq a}\cap C^{PM}(\II).$$ 

The main results of this paper are as follows:

\begin{theorem}\label{thm:main1}
	For every $a\in [0,+\infty)$,
	both $ E_{\geq a}$ and  $E_{>a}$ are homeomorphic to $l_2$.
\end{theorem}

\begin{theorem}\label{thm:main2}
	There exists a homotopy $H:C(\II)\times\II\to C(\II)$  such that 
	\begin{enumerate}
		\item $H_0=\id_{C(\II)}$;
		\item  $\htop(H_t(f))\leq \htop(f)$ and $H_t(f)\in C^{PM}(\II)$   for any $t\in(0,1)$ and $f\in C(\II)$;
		\item $H_1(f)\equiv 0$ for any $f\in C(\II)$.
	\end{enumerate}
\end{theorem}

Restricting  the homotopy in Theorem~\ref{thm:main2} to $E_{\leq a}$ and $E_{<a}$ respectively, we can obtain the following corollary:
\begin{corollary} \label{cor:E-ls-a}
	For every $a\in [0,+\infty]$,
	$E_{\leq a}$ ($E_{<a}$, respectively) is contractible
	and $E_{\leq a}^{PM}$ ($E_{<a}^{PM}$, respectively)
	is homotopy dense in  $E_{\leq a}$ ($E_{<a}$, respectively).
\end{corollary}

The paper is organized as follows.
In Section 2, we recall some basic notions which we will use in the paper.
Theorems~\ref{thm:main1} and \ref{thm:main2} are proved
in Sections 3 and 4 respectively.

\section{Preliminaries}

In this section,  we recall some notions and aspects of infinite-dimensional topology and topological entropy which will be used later.

\subsection{Infinite-dimensional topology}
In this subsection,
we  give some concepts and facts on
general topology and infinite-dimensional topology.
For more information, we refer the reader to \cite{Englking-book-1989-b,Chigogidze-2001,van-book-1989,van-book-2001}.

Let $(X,d)$ be a  metric space.
We say that 
\begin{itemize}
\item $X$ is \textbf{nowhere locally compact}
if no non-empty open set  in $X$ is locally compact;
\item $X$ is an \textbf{absolute (neighborhood) retract}
(\textbf{A(N)R}, briefly) if for every metric space $Y$ which contains $X$ as a closed subspace, there exists a continuous map $r:Y\to X$ ($r:U\to X$ from a neighborhood $U$ of $X$) such that $r|_X=\id$;
\item $X$ has
the \textbf{strong discrete approximation property}
(\textbf{SDAP}, briefly) if for every continuous map $\varepsilon:X\to (0,1)$, every compact metric space $K$ and every continuous map $f:K\times \N\to X$, there exists a continuous map $g:K\times\N\to X$ such that
$\{g(K\times\{n\}):n\in\N\}$ is discrete in $X$ and $d(f(k,n),g(k,n))<\varepsilon(f(k,n))$ for every $(k,n)\in K\times \N$.
\end{itemize}

A \textbf{homotopy} on $X$ is a continuous map $H\colon X\times \II\to X$, $(x,t)\mapsto H_t(x)$.
The space $X$ is said to be \textbf{contractible} if there exists a homotopy $H:X\times\II\to  X$ such that $H_0=\id_X$ and $H_1$ is  a constant map.
A subset $A$ of $X$ is called \textbf{homotopy dense} if there
exists a homotopy $H:X\times\II\to X$ such that $H_0=\id_X$ and
$H_t(x)\in A$  for every $x\in X$ and $t\in (0,1]$.

We will need the following important results in infinite-dimensional topology.

\begin{prop}\label{prop:AR-ANR}
	(\cite[Theorem 5.2.15]{van-book-1989})
	A metric space is an AR if and only if it is a contractible ANR. 	
\end{prop}

\begin{theorem}\label{thm:AN-SDAP-homotopy-dense}
	({\cite[1.2.1 Proposition and Exercise 1.3.4]{Banakh}})
	Let $Y$ be a homotopy dense subspace of $X$.
	If $X$ is an ANR (with SDAP) then $Y$ is also an ANR (with SDAP) .
\end{theorem}

\begin{theorem} \label{thm:AR-SDAP-l2}
	({\cite[1.1.14 (Characterization Theorem)]{Banakh}})
	A separable topologically complete metric space is
	homeomorphic to $l_2$ 
	if and only if it is an AR with SDAP.
\end{theorem}

\begin{theorem} \label{thm:homeo-l2}
	({\cite[5.5.2 Corollary]{Banakh}})
	A convex subspace $X$ of a separable Banach space
	is homeomorphic to $l_2$
	if and only if $X$ is
	topologically complete and nowhere locally compact.
\end{theorem}

The following result must be ``folklore'',
but we can not find a proper reference and therefore we provide a proof
for the completeness.

\begin{prop}\label{prop:CI-l2}
	The function space  $C(\II)$ is homeomorphic to $l_2$.
\end{prop}
\begin{proof}
	Let $C(\II,\R)$ be the collection of all continuous maps from $\II$ to $\R$ with the standard linear structure and the supremum norm.
	Then $C(\II,\R)$ is a separable Banach space.
	The space $C(\II)$ is a closed and convex subspace of $C(\II,\R)$.
	It is not hard to verify that $C(\II)$ is nowhere locally compact.
	It follows from Theorem~\ref{thm:homeo-l2} that $C(\II)$ is homeomorphic to $l_2$. 
\end{proof}

Combining the above results, we have the following useful criterion when a subspace of $C(\II)$ is homeomorphic to $l_2$.
\begin{corollary}\label{cor:CI-subspace}
	A homotopy dense subspace $A$ of $C(\II)$ is homeomorphic to $l_2$
	if and only if it is topologically complete and contractible.
\end{corollary}
\begin{proof}
	The necessity is clear and we only need to prove the sufficiency.  
	By Proposition~\ref{prop:CI-l2}, $C(\II)$ is homeomorphic to $l_2$.
	So by Theorem~\ref{thm:AR-SDAP-l2}, $C(\II)$ is an ANR with SDAP.
	Since $A$ is homotopy dense in $C(\II)$, it follows from Theorem~\ref{thm:AN-SDAP-homotopy-dense} that $A$ is also an ANR with SDAP.
 By the assumption we have $A$ is contractible, then by Proposition~\ref{prop:AR-ANR},
	$A$ is an AR. Finally by Theorem~\ref{thm:AR-SDAP-l2} again,
	$A$ is homeomorphic to $l_2$. 
\end{proof}

\subsection{Topological entropy}

Let $X$ be a compact metric space.
Denote by $Cov(X)$ the family of all open covers of $X$.
For $\alpha, \beta\in Cov(X)$ and $f\in C(X)$,
let
$$N(\alpha)=\min\Bigl\{n\in\mathbb{N}\colon
\mbox{there exist }U_1,U_2,\cdots, U_n\in\alpha\ \mbox{such that}\ \bigcup_{i=1}^nU_i=X\Big\};$$
$$\alpha\vee\beta=\{U\cap V:U\in\alpha, V\in\beta\},\
f^{-1}(\alpha)=\{f^{-1}(U):U\in\alpha\}$$
and
$$\htop(f,\alpha)= \lim\limits_{n\to\infty}
\frac{\log N(\alpha\vee f^{-1}(\alpha)\vee\cdots\vee f^{-n+1}(\alpha))}{n}.$$
The  \textbf{topological entropy} of a continuous map $f:X\to X$ is defined as 
$$\htop(f)=\sup\{\htop(f,\alpha):\alpha\in Cov(X)\}.$$

Let $f\in C(\II)$. A family $\{J_1,J_2,\cdots,J_n\}$ of non-degenerate  closed intervals  is called an \textbf{$n$-horseshoe} if
\begin{enumerate}
	\item $\Int(J_i)\cap \Int(J_j)=\emptyset$ for all $1\leq i<j\leq n $, where $\Int(J_i)$ is the interior of $J_i$ in $\II$;
	\item  $J_i\subset f(J_j)$ for all $1\leq i,j\leq n $.
\end{enumerate}
The following result can be easily obtained, see e.g.\ \cite[Proposition VIII.8]{Block-book-1992}.

\begin{lemma}\label{lem:horseshoe}
	If $f\in C(\II)$ has an $n$-horseshoe, then $h(f)\geq \log n.$
\end{lemma}

The following result was first proved by Misiurewicz,
see e.g.~\cite[Proposition VIII.30]{Block-book-1992}.

\begin{theorem}\label{thm:entropy-LSC}
	The entropy function $\htop\colon C(\II)\to [0,+\infty]$,
	$f\mapsto \htop(f)$ is lower-semicontinuous.
\end{theorem}

\begin{corollary}\label{cor:Ea-open}
	For every $a\in [0,+\infty)$, $E_{>a}$ is open and $E_{\geq a}$ is a $G_\delta$-set in $C(\II)$. 
\end{corollary}

The convexity of $C(\II)$ in the Banach space $C(\II,\R)$ plays
a key role in the proof of Proposition~\ref{prop:CI-l2}.
The following examples show that neither $E_{\leq a}$ nor $E_{>a}$ is  convex in  $C(\II,\R)$.

\begin{example}
	Note that for every $f\in C(\II)$,
	if $f(\frac{1}{2}-x)=f(\frac{1}{2}+x)$ for all $x\in [0,\frac{1}{2}]$.
	then $f$ and $1-f$ are topologically conjugate and thus $\htop(1-f)=\htop(f)$.
	But $\htop(\frac{1}{2}f+\frac{1}{2}(1-f))=\htop(\frac{1}{2})=0$.
	It follows that
	$E_{>a}$ is not convex for any $a\in [0,+\infty)$.
\end{example}

\begin{example}
	It is well-known that for every $f\in C(\II)$,
	$\htop(f)=0$ if and only if all periods  of $f$ are of the form $2^n$
	(see e.g.\ Proposition~VIII.34 and Theorem II.14   in \cite{Block-book-1992}).
	Let $f$ and $g$ be the broken line maps through
	the points $(0,1), (\frac{1}{4},0),\ (1,0)$ and
	the points  $(0,\frac{1}{2}),\ (\frac{1}{4},0),\ (\frac{1}{2},0),\ (\frac{3}{4}, \frac{1}{2}),\ (1,\frac{1}{2})$, respectively.
	Then it is not hard to verify
	that   $n$ is a period for $f$ or $g$ if and only if $n=1$ or $2$.
	It follows that $\htop(f)=\htop(g)=0$.
	For the convex combination
	$\varphi=\frac{1}{2}f+\frac{1}{2}g$,
	we have $\varphi(0)=\frac{3}{4}$,  $\varphi(\frac{3}{4})=\frac{1}{4}$ and $\varphi(\frac{1}{4})=0 $.
	It follows that $0$ is a periodic point with period  $3$ for $\varphi$, which implies $\htop(\varphi)>0$.
	This shows that $E_{\leq 0}$ is not convex.
\end{example}

\section{Proof Theorem~\ref{thm:main1}}
In this section, we will prove Theorem~\ref{thm:main1}.
At first, we introduce the box maps defined in \cite{Kolyada-2151}. Define a subset $\Lambda$ of $\mathbb{R}^5$ as follows
$$\Lambda=\{(a_l,a_r,a_b,a_t,a_s)\in\mathbb{R}^5:a_b<a_t,a_l,a_r\in [a_b,a_t], a_s\geq 20\}.$$
For every non-degenerate closed interval $K=[a_0,a_1]$ and $\lambda=(a_l,a_r,a_b,a_t,a_s)\in\Lambda$, the authors in \cite{Kolyada-2151} defined a continuous surjection
$\xi_\lambda:K\to [a_b,a_t]$, which was called a \textbf{box map}, such that  $\xi_\lambda$ is piecewise linear with constant slope $\frac{a_s(a_t-a_b)}{a_1-a_0}$, $\xi_\lambda(a_0)=a_l$ and $\xi_\lambda(a_1)=a_r$.
We make this construction both from left and right,
$\xi_\lambda$ is increasing on the leftmost lap unless $a_l=a_t$
and decreasing on the rightmost one unless $a_r=a_t$.
We choose the \textbf{meeting point $m$} to be on the fifth decreasing lap from the left (see Figure~\ref{fig:box-map}
for example).
If the left and right graphs coincide, then there is no
well-defined meeting point, but the graph of $\xi_\lambda$ is clear.

\begin{figure}[htbp!]
	\centering
	\includegraphics[scale=0.3]{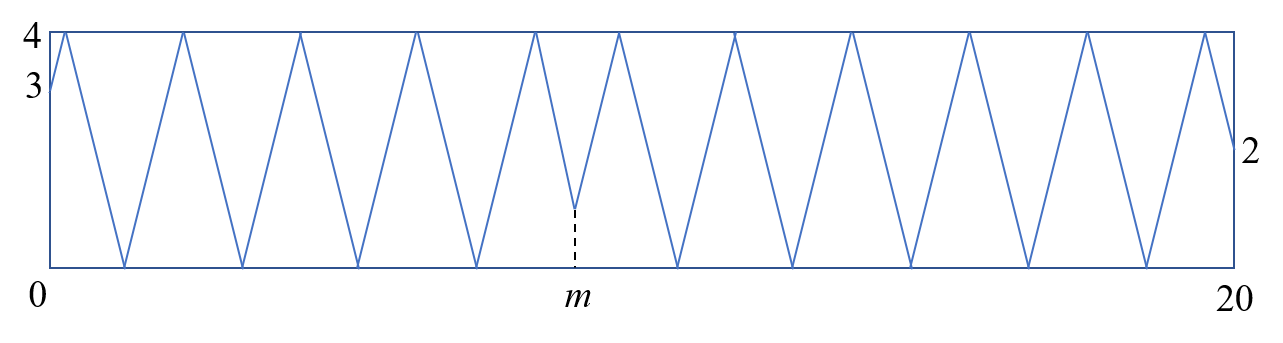}
	\caption{$K=[0,20]$, $a_l=3$, $a_r=2$, $a_b=0$, $a_t=4$, $a_{s}=20$}
	\label{fig:box-map}
\end{figure}

\begin{remark}\label{rem:box-map-entropy}
	Let $\xi_\lambda$ be a box map on $K$.
	If $a_b=a_0$ and $a_t=a_1$, then there exist 
	closed subintervals $J_1$, $J_2$, $\dotsc$, $J_{[a_s-4]}$ of $K$ with disjoint interiors such that
	$f (J_j)=K$ for $j=1,2,\dotsc, [a_s-4]$, 
	where $[x]$ is the greatest integer less than or equal to  $x$.
	Hence, $J_1$, $J_2$, $\dotsc$, $J_{[a_s-4]}$
	form an $[a_s-4]$-horseshoe of $\xi_\lambda$.
	By Lemma~\ref{lem:horseshoe}, $\htop(\xi_\lambda)\geq \log([ a_s-4])$.
\end{remark}

Following the idea in \cite{Kolyada-2151},
for every $\alpha\ge 20$ we first construct a homotopy 
$\widetilde{H}^{\alpha}:C(\II)\times\II\to C(\II)$ as follows.
Fix a function $f\in C(\II)$.
First let $\widetilde{H}^{\alpha}_0(f)=f$.
For $t\in (0,1]$,
let $s$ be the largest non-negative integer such that $st<1$.
We obtain $s+1$ closed intervals:
$$I_i=[(i-1)t,it],\ i=1,2,\cdots,s,\ I_{s+1}=[st,1].$$
In particular, if $t=1$, then $s=0$ and we have only one closed interval $I_1=[0,1]$. For $i=1,2,\cdots,s+1$, let
$\alpha_i=\max\{|I_i|, |f(I_i)|\}$, where $|J|$ is the length of a closed interval $J$,
and
\begin{align*}
a^i_b&=\max\{0,\min f(I_i)-4\alpha_i\}; \\
a^i_t&=\min\{1,\max f(I_i)+4\alpha_i\}; \\
a^i_l&=f(\min I_i);\\
a^i_r&=f(\max I_i).
\end{align*}
It is not hard to verify that if $I_i\cap f(I_i)\not=\emptyset$
then
\begin{equation}\label{1}
I_i\subset [a^i_b,a^i_t].
\end{equation}
It is clear that $\lambda_i^\alpha=(a^i_l,a^i_r,a^i_b,a^i_t,\alpha)\in \Lambda$
and then we define $\widetilde{H}^{\alpha}_t(f)$ on $I_i$ as the box map $\xi_{\lambda_i^\alpha}\in C(I_i,\II)$.
So $H_t^\alpha(f)$ is well-defined for $t\in (0,1]$.
By Lemma 2.2 of \cite{Kolyada-2151},
$\widetilde{H}^{\alpha}:C(\II)\times\II\to C(\II)$ is a homotopy.
Note that $\widetilde{H}^{\alpha}_0=\id_{C(\II)}$ and for every $f\in C(\II)$,
$\widetilde{H}^{\alpha}_1(f)$ is the box map on $\II$ with the parameter $(f(0),f(1),0,1,\alpha)$.
Now we construct another homotopy
$\widehat{H}^{\alpha}\colon \widetilde{H}^{\alpha}_1(C(\II))\times \II\to C(\II)$. For every $f\in C(\II)$ and $t\in[0,1]$ we define
$\widehat{H}^{\alpha}_t(\widetilde{H}^{\alpha}_1(f))$
to be the box map on $\II$ with the parameter $((1-t)f(0),(1-t)f(1),0,1,\alpha)$.
By Lemma 2.1 of \cite{Kolyada-2151}, $\widehat{H}^{\alpha}$
is continuous then it is a homotopy. 
It should be noticed that for every $f\in C(\II)$,
$\widehat{H}^{\alpha}_1(\widetilde{H}^{\alpha}_1(f))$ is the 
box map on $\II$ with the parameter $(0,0,0,1,\alpha)$.
Finally, we define a homotopy $H^\alpha\colon C(\II)\times\II\to C(\II)$ by joining $\widehat{H}^{\alpha}$ and
$\widetilde{H}^{\alpha}$, that is, for every $f\in C(\II)$,
$H_t^\alpha(f)=\widetilde{H}_{2t}^{\alpha}(f)$ for $t\in [0,\frac{1}{2}]$
and $H_t^\alpha(f)=\widehat{H}_{2(t-\frac{1}{2})}^{\alpha}(\widetilde{H}_{1}^{\alpha}(f))$ for $t\in (\frac{1}{2},1]$.

We have the following estimation of the topological entropy of $H_t^\alpha(f)$.
\begin{lemma}
	For every $t\in (0,1]$, $\alpha\ge 20$ and $f\in C(\II)$,
	we have
	\[
	\htop(H^\alpha_t(f))\geq \log([ \alpha-4]).
	\]
\end{lemma}
\begin{proof}
	Fix $\alpha\ge 20$ and $f\in C(\II)$.
	By Remark~\ref{rem:box-map-entropy},
	we have $\htop(H^\alpha_t(f))\geq \log([ \alpha-4])$ for all 
	$t\in [\frac{1}{2},1]$.
	Now assume that $t\in[0,\frac{1}{2})$.
	By the construction of $H^\alpha_t$,
	there exists an interval $I_i$ and $x_0\in I_i$ such that $f(x_0)=x_0$.
	By the formula \eqref{1},
	we have $I_i\subset [a^i_b,a^i_t]$.
	Now by the construction of the box map on $I_i$,
	there exist closed subintervals $J_1$, $J_2$, $\dotsc$, $J_{[ \alpha-4]}$ of $I_i$ with disjoint interiors such that
	$H_t^\alpha(f) (J_j)=[a^i_b,a^i_t]$ for $j=1,2,\dotsc, [ \alpha-4]$.
	Then $J_1$, $J_2$, $\dotsc$, $J_{[ \alpha-4]}$
	form an $[ \alpha-4]$-horseshoe of $H_t^\alpha(f)$.
	By Lemma~\ref{lem:horseshoe}, $\htop(H^\alpha_t(f))\geq \log([ \alpha-4])$. 
\end{proof}

We summarize the above results as follows.
\begin{prop}\label{prop:homotopy-entropy}
	For every $\alpha\ge 20$, there exists
	a homotopy $H^\alpha:C(\II)\times\II\to C(\II)$ such that
	\begin{enumerate}
		\item $H_0^\alpha=\id_{C(\II)}$;
		
		\item  $\htop(H_t^\alpha(f))\geq\log([\alpha-4])$ for $t\in (0,1]$ and for every $f\in C(I)$;
		\item
		$H^{\alpha}_1(f)$ is the 
		box map on $\II$ with the parameter $(0,0,0,1,\alpha)$.
	\end{enumerate}
\end{prop}

Now  we are ready to prove Theorem \ref{thm:main1}.
\begin{proof}[Proof of Theorem \ref{thm:main1}]
	Fix $a\in[0,+\infty)$ and choose $\alpha\in[20,+\infty)$ such that
	$\log([ \alpha-4]) >a$.
	Let $H^\alpha$ as in the Proposition~\ref{prop:homotopy-entropy}.
	Then  both $E_{\geq a}$ and $E_{>a}$ are homotopy dense in $C(\II)$.
	Using the homotopies $ H^\alpha|_{E_{\geq a}\times\II} $ and $ H^\alpha|_{E_{>a} \times\II}$, both $E_{\geq a}$ and $E_{>a}$ are contractible.
	By Corollary~\ref{cor:Ea-open},
	both $E_{\geq a}$ and $E_{>a}$ are topologically complete.
	Now using Corollary~\ref{cor:CI-subspace},  $E_{\geq a}$ and $E_{>a}$
	are homeomorphic to $l_2$. 
\end{proof}

\begin{corollary}\label{cor:homotopy-dense}
	For every $a\in [0,+\infty)$, $E_{\geq a}$ and $E_{>a}$ are homotopy dense in $C(\II)$.
	Moreover, $E_{>a}\cap E_{<+\infty}$ is homotopy dense and open in $E_{<+\infty}$.\end{corollary}
\begin{proof}
	The former was shown in the proof of Theorem~\ref{thm:main1}. To show the latter, we only note that the
	topological entropy of a piecewise monotone map is finite (see e.g. \cite[Proposition VIII.18]{Block-book-1992}).
\end{proof}

By Theorem~\ref{thm:entropy-LSC}
and Corollary~\ref{cor:homotopy-dense},
we know that the subspace $E_{+\infty}=\{f\in C(\II)\colon \htop(f)=+\infty\}$
is a dense $G_\delta$-set in $C(\II)$.
But the following question remains open.

\begin{problem}
	Is $E_{+\infty}$ homeomorphic to $l_2$?
\end{problem}

In Proposition~\ref{prop:homotopy-entropy},
for every $f\in C(\II)$ and $t\in (0,1]$, $H^\alpha_t(f)$ is  piecewise monotone and then  it has finite topological entropy.
So we can not use the method in the beginning of this section
to construct a proper homotopy to show that $E_{+\infty}$ is contractible.
Another important fact is that there is no continuous selection
of fixed points.
\begin{prop}
	There does not exist a continuous map $\phi:C(\II)\to\II$ such that $\phi(f)$ is a fixed point of $f$ for every $f\in C(\II).$
\end{prop}
\begin{proof}
	Suppose that $\phi:C(\II)\to\II$ is such a map.
	Choose $x_0\in \II\setminus\{0,1,\phi(\id_{\II})\}.$
	Let $l_n:\II\to\II$ be the broken line map through the points
	$(0,\frac{1}{n})$, $(x_0,x_0)$ and $(1,1-\frac{1}{n})$.
	Then $l_n\to \id_{\II}$ in $C(\II)$ as $n\to\infty$.
	Since $l_n$ has a unique fixed point $x_0$,
	$\phi(l_n)=x_0\not\to \phi(\id_{\II})$ as $n\to\infty$.
	So $\phi$ is not continuous, which is a contradiction.
\end{proof}

\section{Proof of Theorem~\ref{thm:main2}}
In this section we construct the homotopy in Theorem~\ref{thm:main2},
which is done by connecting three homotopies.

Inspired by \cite{Grinc} and \cite{JS91}, we introduce the following concept.
Let $f$ and $\bar f\in C([a,b],\II)$.
We say that $\bar f$ is made from $f$ by  {\bf procedure of making constant pieces} ({\bf PMCP}, briefly)
if there exists a sequence of open intervals $\{U_n\}_{n=1}^\infty$ of $[a,b]$ in the relative topology such that 
$\bar f|_{I\setminus\bigcup_{n=1}^\infty U_n}=  f|_{I\setminus\bigcup_{n=1}^\infty U_n}$ and $\bar f|_{U_n}$ is constant 
for every $n\in\mathbb{N}$.
It should be noticed that our definition here is more general than the one in \cite{Grinc}.
We will need  the following result which was proved in \cite[Lemma 5]{JS91}.
\begin{lemma}\label{lem 4}
	Let $f\in C(\II)$.
	If $\overline{f}$ is made from $f$ by  PMCP,
	then $\htop(\overline{f})\leq \htop(f).$
\end{lemma}
For every $c\in\II$, the map $\max\{f(x),c\}$ can be thought to be made from $f$ by PMCP.
For every $f\in C([a,b],\II)$, let
\begin{align*}
M(f)&=\max\{f(x):x\in [a,b]\};\\
c_1(f)&=\min\{x\in [a,b]:f(x)=M(f)\};\\
c_2(f)&=\max\{x\in [a,b]:f(x)=M(f)\}.
\end{align*}
Now we define $\widetilde{f}:[a,b]\to \II$ as follows
\begin{equation*}
\ww{f}(x) =
\begin{cases}
\max\{f(t):a\leq t\leq x\}, & x\in [a,c_1(f)],\\
M(f), & x\in [c_1(f),c_2(f)],\\
\max\{f(t):x\leq t\leq b\}, & x\in [c_2(f),b].
\end{cases}
\end{equation*}
First we have the following lemma.
\begin{lemma}\label{lem 5}
	For any $f,g\in C([a,b],\II)$, we have
	\begin{enumerate}
		\item $\ww{f}$ is made from $f$ by  PMCP and it is in $C^{PM}([a,b],\II)$;
		
		\item  $\ww{f}(a)=f(a)$, $\ww{f}(b)=f(b)$ and $\ww{f}([a,b])\subset f([a,b])$;
		
		\item  $d(\ww{f},\ww{g})\leq d(f,g)$;
		
		\item if $c\in (a,b)$ and $\varepsilon>0$ satisfy either
		$$\max f|_{[a,c]}-\min f|_{[a,c]}<\varepsilon\ \ \mbox{or}\ \ \max f|_{[c,b]}-\min f|_{[c,b]}<\varepsilon,$$
		that is, the  amplitude of $f$ on $[a,c]$ or on $[c,b]$ is smaller than $\varepsilon$,
		then $$d(\ww{f},\ww{f|_{[a,c]}}\cup \ww{f|_{[c,b]}})<\varepsilon.$$
	\end{enumerate}
	
\end{lemma}
\begin{proof}
	(1) and (2) are obvious. We only need to show (3) and (4).
	
	(3)  We note that for any maps $h,k:J\to\II$,
	\begin{equation}\label{not=}
	\bigl|\sup \{h(x):x\in J\}-\sup \{k(x):x\in J\}\bigr|
	\leq\sup\{|h(x)-k(x)|:x\in J\}.
	\end{equation}
	It follows that (3) holds in the case
	$[c_1(f),c_2(f)]\cap [c_1(g),c_2(g)]\not=\emptyset$.
	For the case $[c_1(f),c_2(f)]\cap [c_1(g),c_2(g)]=\emptyset$, without loss of generality, we assume that $c_2(f)<c_1(g)$.
	For $x\in [a,b]\setminus [c_2(f),c_1(g)]$ using
	the formula \eqref{not=}, we have that
	$$|\ww{f}(x)-\ww{g}(x)|\leq d(f,g).$$
	If $x\in (c_2(f),c_1(g))$ and $\ww{f}(x)\geq\ww{g}(x)$, then
	$$0\leq\ww{f}(x)-\ww{g}(x)\leq f(c_2(f))-g(c_2(f))\leq d(f,g).$$
	If $x\in (c_2(f),c_1(g))$ and $\ww{g}(x)>\ww{f}(x)$, then
	$$0<\ww{g}(x)-\ww{f}(x)\leq g(c_1(g))-f(c_1(g))\leq d(f,g).$$
	Hence (3) holds in the case $[c_1(f),c_2(f)]\cap [c_1(g),c_2(g)]=\emptyset$.
	
	(4)  Without loss of generality,
	we assume that $\max f|_{[a,c]}-\min f|_{[a,c]}<\varepsilon$.
	By (2), $h=\ww{f|_{[a,c]}}\cup \ww{f|_{[c,b]}}\in C([a,b],\II)$.
	
	\textbf{Case A}: $c\in [c_1(f),c_2(f)]$.
	By the assumption, we have $M(f)-\varepsilon<f(c)\leq M(f)$.
	It follows that
	$$M(f)-\varepsilon< f(c)\leq h(x)\leq M(f)=\ww{f}(x),\ \ x\in [c_1(f),c_2(f)].$$
	Hence
	$$|\ww{f}(x)-h(x)|<\varepsilon,\ \  x\in [c_1(f),c_2(f)].$$
	Moreover, it is trivial that $\ww{f}(x)=h(x)$ for $x\in [a,b]\setminus [c_1(f),c_2(f)]$. Hence
	$$d(\ww{f},h)<\varepsilon.$$
	
	\textbf{Case B}: $c\in [a,c_1(f)]$. In this case,
	$$\ww{f}(x)=h(x),\ \ \ x\in [c_1(f),b].$$
	Moreover, by the assumption in (4),
	$$|\ww{f}(x)-h(x)|<\varepsilon,\ \ \  x\in [a,c].$$
	Furthermore,  for every $x\in [c,c_1(f)]$,
	$$h(x)=\ww{f|_{[c,b]}}(x)\leq \ww{f}(x)< \ww{f|_{[c,b]}}(x)+\varepsilon=h(x)+\varepsilon.$$
	Therefore, $$d(\ww{f},h)<\epsilon.$$
	
	\textbf{Case C}: $c\in [c_2(f),b]$.
	By the assumption, $M(f)-\varepsilon<f(c)\leq M(f)$.
	It follows that
	\begin{equation}\label{2}
	M(f)-\varepsilon<f(c)\leq f(c_2(f|_{[c,b]}))\leq M(f).
	\end{equation}
	Note that
	$$\ww{f}(x)=\ww{f|_{[c,b]}}(x)\leq f(c_2(f|_{[c,b]})),\ \  x\in [c_2(f|_{[c,b]}),b].$$
	Moreover, using this and the formula (\ref{2}), we have
	\[|\ww{f}(x)-h(x)|<\varepsilon,\ \ \ x\in [a, c_2(f|_{[c,b]})].\]
	So in this case we also have $d(\ww{f},h)<\varepsilon $.
\end{proof}

Using the above, we can give the  first homotopy.
\begin{lemma}\label{lem:homotopy1}
	There exists a homotopy $H^1:C(\II)\times \II\to C(\II)$ such that
	\begin{enumerate}
		\item $H^1_0=\id_{C(\II)}$;
		
		\item  $\htop(H^1_t(f))\leq \htop(f)$ and
		$H^1_t(f)\in C^{PM}(\II)$ for $t\in (0,1]$ and $f\in C(\II)$.
	\end{enumerate}
\end{lemma}
\begin{proof}
	In the same way as in the construction of the homotopy $H^\alpha$ in Section 3, let $H^1_0=\id_{C(\II)}$, and for $t\in (0,1]$, let $s$ be the largest non-negative integer
	such that $st<1$.
	We can obtain $s+1$ closed intervals:
	$$I_i=[(i-1)t,it], i=1,2,\cdots,s, I_{s+1}=[st,1].$$
	The integer $s$ and the interval $I_i$ are also denoted by $s(t)$ and $I^t_i$ if necessary. We define $H_t^1$ such that, for every $f\in C(\II)$ and $i=1,2,\dotsc,s+1$,
	$$H_t^1(f)|_{I_i}=\ww{f|_{I_i}}.$$
	Using Lemma \ref{lem 5}(2),
	$H^1:C(\II)\times \II\to C(\II)$ is well-defined.
	Trivially, it satisfies (1).
	From Lemmas \ref{lem 4} and \ref{lem 5}(1)
	it follows that it satisfies (2).
	It remains to verify that
	$H^1:C(\II)\times \II\to C(\II)$ is continuous.
	
	At first, we show that $H^1(f,\cdot\,)$ is continuous for every fixed $f\in C(\II)$. For every $\varepsilon>0$, there exists $\delta\in (0,1)$ such that
	\begin{equation}\label{u-con}
	|x_1-x_2|<\delta\ \ \ \mbox{implies}\ \ \ |f(x_1)-f(x_2)|<\frac{\varepsilon}{2}.
	\end{equation}
	Now, for every $t_0\in\II$, we verify that there exists $\delta(t_0)\in (0,\delta]$ such that
	\begin{equation}\label{pur}
	|t-t_0|<\delta(t_0)\ \ \ \mbox{implies}\ \ \ d(H^1(f,t),H^1(f,t_0))<\varepsilon,
	\end{equation}
	which shows that $H^1(f,\cdot\,)$ is continuous.
	
	If $t_0=0$, we let $\delta(t_0)=\delta$.
	For every $t\in (0,\delta)$ and $i$, from Lemma \ref{lem 5}(2)
	it follows that
	$$\ww{f|_{I_i}}(I_i)\subset f(I_i).$$
	Since $|I_i|\leq t<\delta$,
	using the formula (\ref{u-con}), we have $|f(I_i)|<\varepsilon$.
	Thus the formula (\ref{pur}) holds.
	
	If $t_0\in (0,1]$, choose $\delta(t_0)\in (0,\delta)$ small enough such that for every $t\in \II\cap (t_0-\delta(t_0),t_0+\delta(t_0))$,
	we have $|s(t_0)-s(t)|<2 $ and  $(s(t_0)+2)\delta(t_0)<\delta$.
	Then all points $\{it,jt_0\}$ divide $\II$
	into closed intervals $\{J_j\}$.
	Let
	$$G=\bigcup \ww{f|_{J_j}}\in C(\II).$$
	Then, for every $i$, $I^{t_0}_i$ is either a union of the two closed intervals in  $\{J_j\}$  or just a closed interval in $\{J_j\}$. If the former holds, then by the choice of $\delta(t_0)$ and the formula (\ref{u-con}), the amplitude
	of $f$ in one of the two closed intervals
	is smaller than $\frac{\varepsilon}{2}$.
	Using Lemma \ref{lem 5}(4), we have that $$d(H^1(f,t_0)|_{I_i^{t_0}},G|_{I_i^{t_0}})<\frac{\varepsilon}{2}.$$
	If the later holds, then $H^1(f,t_0)|_{I_i^t}=G|_{I_i^t}$
	and hence the above formula also holds. Thus,
	$$d(H^1(f,t_0),G)<\frac{\varepsilon}{2}.$$
	Similarly, we have that
	$$d(H^1(f,t),G)<\frac{\varepsilon}{2}.$$
	Hence the formula (\ref{pur}) holds.
	
	By Lemma \ref{lem 5}(3),
	we can obtain that $d(H^1(f,t),H^1(g,t))\leq d(f,g)$.
	In combination with the continuity of $H^1$ on $t$,
	we have that $H^1:C(\II)\times \II\to C(\II)$ is jointly continuous.
\end{proof}

The second homotopy we need is the following.
\begin{lemma}\label{lem:homotopy2}
	There exists a homotopy $H^2:C(\II)\times \II\to C(\II)$ satisfying
	\begin{enumerate}
		\item $H^2_0=\id_{C(\II)}$;
		
		\item  $\htop(H^2_t(f))\leq \htop(f)$ and $H^2_t(C^{PM}(\II))\subset C^{PM}(\II)$ for any $t\in (0,1]$ and $f\in C(\II)$;
		
		\item  $H^2_1(f)$ is a constant map for any $f\in C(\II)$.
	\end{enumerate}
\end{lemma}
\begin{proof}
	For every $f\in C(\II)$, let
	$$M(f)=\max \{f(x):x\in \II\},\ \ \ \ m(f)=\min \{f(x):x\in \II\}.$$
	Then $M,m:C(\II)\to\II$ are continuous. Using them,  we can define our homotopy as follows
	$$ H^2(f,t)(x) =\max\{f(x), (M(f)-m(f))t+m(f))\}.$$
	Then it is not hard to verify that $H^2:C(\II)\times\II\to C(\II)$ is continuous and it satisfies (1) and (3).
	Moreover, $H^2(f,t)$ is made from $f$ by PMCP.
	It follows from Lemma~\ref{lem 4} that $H^2$ also satisfies (2).
\end{proof}

The third homotopy $H^3:\II^2\to\II$ is defined as  
$$H^3(s,t)=(1-t)s,$$ which is a homotopy between the identical map and the constant map $0$ in $\II$.
Now we are ready to prove Theorem \ref{thm:main2}.
\begin{proof}[Proof of Theorem \ref{thm:main2}]
	Define $H:C(\II)\times\II\to C(\II)$ as
	\begin{equation*}
	H(f,t) =
	\begin{cases}
	H^1(f,3t),& t\in [0,\frac{1}{3}),\\
	H^2(H^1(f,1),3t-1), & t\in [\frac{1}{3},\frac{2}{3}),\\
	H^3(H^2(H^1(f,1),1),3t-2), & t\in [\frac{2}{3},1].
	\end{cases}
	\end{equation*}
	Since $H^2(H^1(f,1),1)$ is a constant map,
	the homotopy $H$ is well-defined.
	Note that $\htop(c)=0$ for every constant map $c$.
	By Lemmas~\ref{lem:homotopy1} and \ref{lem:homotopy2},
	it is easy to see that
	$H:C(\II)\times\II\to C(\II)$ is the homotopy as required.
\end{proof}

It follows from Corollary \ref{cor:homotopy-dense} and Theorem~\ref{thm:entropy-LSC} that $E_{\le a}$ is nowhere dense and closed in the space $E_{<+\infty}$.
Hence
$$E_{<+\infty}=\bigcup_{n\in\N}E_{\le n}$$
is not topologically complete.
Therefore, $E_{<+\infty}$ is not homeomorphic to $l_2$.
It is natural to put the following problem:
\begin{problem}
	Does there exist $a\in (0,+\infty)$ such that
	$E_{<a}$ is homeomorphic to $l_2$?
\end{problem}
For every $a\in [0,+\infty)$, by Corollary~\ref{cor:E-ls-a}, we know that $E_{\le a}$ is contractible.
By Theorem~\ref{thm:entropy-LSC},
$E_{\le a}$ is a closed subset of $C(\II)$ and hence it is topologically complete. But the following problem is still open.
\begin{problem}
	Is $E_{\le a}$ homeomorphic to $l_2$ for every $a\in [0,+\infty)$? In particular, is $E_{\le 0}$ homeomorphic to $l_2$?
\end{problem}

By Anderson-Kadec's theorem, $l_2$ is homeomorphic to $\R^\N$, then it is also homeomorphic to $s=(-1,1)^\N$.
Let
\begin{align*}
Q&=[-1,1]^\N, \\
\Sigma&=\{(x_n)\in Q:\sup |x_n|<1\},\\
P_{\prec 2^\infty}&=\{f\in C(\II):\mbox{there exists }n \in\mathbb{N}
\mbox{ such that }\\
&\qquad \qquad\qquad\mbox{ the set of periods of}~f~\mbox{is}~ \{2^i:0\leq i\leq n\}\}.
\end{align*}
Using these symbols, we have the following problem:
\begin{problem} For every $a\in (0,+\infty]$, does there exist a homeomorphism $h:E_{\leq a}\to s\times Q$ such that $h(E_{<a})=s\times \Sigma$?
	Does there exist  a homeomorphism $h:E_{\leq 0}\to s\times Q$ such that $h(P_{\prec 2^\infty})=s\times \Sigma$?
\end{problem}

\subsection*{Acknowledgments}
The authors would like to thank the
anonymous referees for the careful reading and helpful suggestions.

\bibliographystyle{amsplain}
 
\providecommand{\bysame}{\leavevmode\hbox to3em{\hrulefill}\thinspace}
\providecommand{\MR}{\relax\ifhmode\unskip\space\fi MR }
\providecommand{\MRhref}[2]{%
	\href{http://www.ams.org/mathscinet-getitem?mr=#1}{#2}
}
\providecommand{\href}[2]{#2}

\end{document}